\documentclass[]{amsart}
\usepackage{amsmath}
\usepackage{amsfonts}
\usepackage{amssymb}
\usepackage{amsthm}
\usepackage{multirow}
\usepackage{graphicx}

\usepackage{mathabx} 
\usepackage{geometry}
\usepackage[T1]{fontenc}
\usepackage[utf8]{inputenc}
\usepackage{float}

\newtheorem{theorem}{Theorem}
\newtheorem{proposition}{Proposition}

\theoremstyle{remark}

\graphicspath{ {figs/} }
\title[Noise induced Stability of a Mean-Field model of Systemic Risk]{Noise induced Stability of a Mean-Field model of Systemic Risk with uncertain robustness}
\author{Alexander Alecio}

\date{\today}

\begin{document}
	
	\begin{abstract}
        We consider a model for systemic risk comprising of a system of diffusion processes, interacting through their empirical mean. Each process is subject to a confining double-well potential with some uncertainty in the coefficients, corresponding to fluctuations in height of the potential barrier seperating the two wells. This is equivalent to studying a single McKean-Vlasov SDE with explicit dependence on its moments and, novelly, independently varying additive and multiplicative noise. Such non-linear SDEs are known to possess two phases: stable (ordered) and unstable (disordered). When the potential is purely bistable, the phase changes from stable to unstable when noise intensity is increased past a critical threshold.
        
		With the recent advances in \cite{alecio}, it will be shown that the behaviour here is far richer: indeed, depending on the interpretation of the stochastic integral, the system exhibits phase changes that cannot occur in any regime where there is no uncertainty in the potential. Strikingly, this allows for the phenomenon of noise induced stability; situations where more noise can reduce the risk of system failure.


	\end{abstract}
    \maketitle

    Consider an evolving system of interconnected components that can transition between two states, functioning or failed. If a sufficient number of individual components were to be in the failed state concurrently, the whole system fails; termed `systemic failure'. Each component has an intrinsic stability, a quantification of its robustness, that competes with a random perturbation that destabilises their state. Interconnectedness (or cooperation) between components, the degree of which can be varied, works to stabilise individual components, assuming a sufficient number of the rest are in a functioning state. The expected trade-off of increasing interconnectedness is an increase in `systemic risk', the probability of systemic failure; see \cite{ssra} for an overview of systemic risk analytics.

    Systemic risk is an important consideration in many fields. The archetypal example from engineering would be a system of interacting components that can cooperate by sharing load, but will sytemically fail if a sufficient number of its constituent components themselves are in the failed state. One tangible realisation are power grids, \cite{pgs}: individual substations may pass demand onto other stations to avoid individual failure, at the risk of total grid failure. Another are banks, which cooperate by lending to one another to prevent default. This linkage is a potential `contagion channel' \cite{lds} as creditor banks are left in a vulnerable position if exposed enough to a defaulting bank. This in turn may lead to further defaults, known as `financial contagion' and documented to have occured in many financial crises \cite{fca, boe, spil}.

	\section{Mean-Field Modeling of Systemic Risk}

    To capture this, a system of $n$ weakly interacting diffusions was introduced in \cite{ss}, where the equation for the risk state (or variable) of component $i$ is
	\begin{equation}\label{ptype}
	dX_t^{n,i}=(-V^{'}(X_t^{n,i})-\theta(X_t^{n,i}-\frac{1}{n}\sum\nolimits_j X_t^{n,j}))dt+\sigma dB_t^i
    \end{equation}
	Each component can either be in a functioning or failed state, corresponding to whether its risk state is positive or negative.  
    Accordingly, the potential $V$ is taken to be symmetric and with minima at $\pm 1$, with the two potential wells seperated by a local maxima at 0. With external perturbation, whose strength is controlled by noise intensity parameter $\sigma$, these minima are metastable: the risk states tend to remain in a potential well, but with a non-zero probability of exit in a finite time. The intrinsic stability, resistance of the components to changing risk state, is encoded in the potential $V$. Cooperation, the degree of which is controlled by $\theta>0$, is expressed through a simple mean reversion mechanism. Systemic failure occurs when a majority of components are themselves in the failed state. Commensurately, as noted in \cite{ss}, the natural choice of measure of systemic risk is the mean of the risk states, $\bar x = \frac{1}{n}\sum_jX^{n,j}$. 
		
	Calculating probabilties of events of $\bar{x}$ is complicated by no closed form forward equation for $\bar x$ existing outside of linear or convex potentials. However, it is known, under certain technical conditions, for instance \cite{leo, ohl}, that the empirical measure of $n$-SDE system (\ref{ptype}) converges on any finite time interval to the solution of the non-linear Fokker Planck equation 
    $$
		\frac{\partial}{\partial t}\rho=\frac{\partial}{\partial x}\Big((-V^{'}(x)-\theta(m_1-x))\rho+\frac{\sigma^2}{2}\frac{\partial\rho}{\partial x}\Big)
	$$
    which is the concomitant forward equation of the McKean-Vlasov process
    \begin{equation}
        \label{vanpr}
		dX_t=(-V^{'}(X_t)-\theta(X_t-m_1))dt+\sigma dW_t
	\end{equation}
    where $m_1=\int x\rho dx$, to which $\bar x(t)$ converges. This is an example of the `propagation of chaos', \cite{Chaintron_2022, sznit}, effectively generalising the problem into a larger space. 

    Behaviour stemming from the explicit dependence on moments in the drift in MV-SDE (\ref{vanpr}) has received much sustained attention, particularly when $V$ is taken as the simple bistable potential, $V=\frac{x^4}{4}-\frac{x^{2}}{2}$ (the Dawson-Shiino model for seminal papers \cite{dawson, shiino}). This includes convergence in different metrics, Central Limit theorem-type result for the fluctuations of $\bar x$ around $\int x\rho dx$, large deviations and (possible) phase transitions and their location, much of which has been extended to arbitrary potentials. 
    
	Idealised macroscopic systems forced from thermodynamic equilibrium eventually undergo a continuous symmetry-breaking instability. Like these instabilities, it has been shown by many authors, \cite{alecio,dawson,shiino,tug}, that MV-SDE (\ref{vanpr}) at stationarity demonstrates almost identical phenomenology to a second order phase transition: once the noise intensity $\sigma$ is pushed beyond a certain critical threshold, $\sigma_c$, the stable (ordered) phase gives way to the unstable (disordered) phase. The stable phase is characterised by three stationary measures (corresponding to the three extrema of $V$), as opposed to the unstable where only one exists. Casting $\sigma$ as the control parameter, admissible stationary solutions have the characteristic property $\mathbb{E}(X_\infty)$ (the mean of process $X_t$ at stationarity) which plays the  r\^ole of order parameter. Plotting these quantities reveals a pitchfork bifurcation. 

		\begin{figure}[t]
		\centering
		\includegraphics[width=\textwidth]{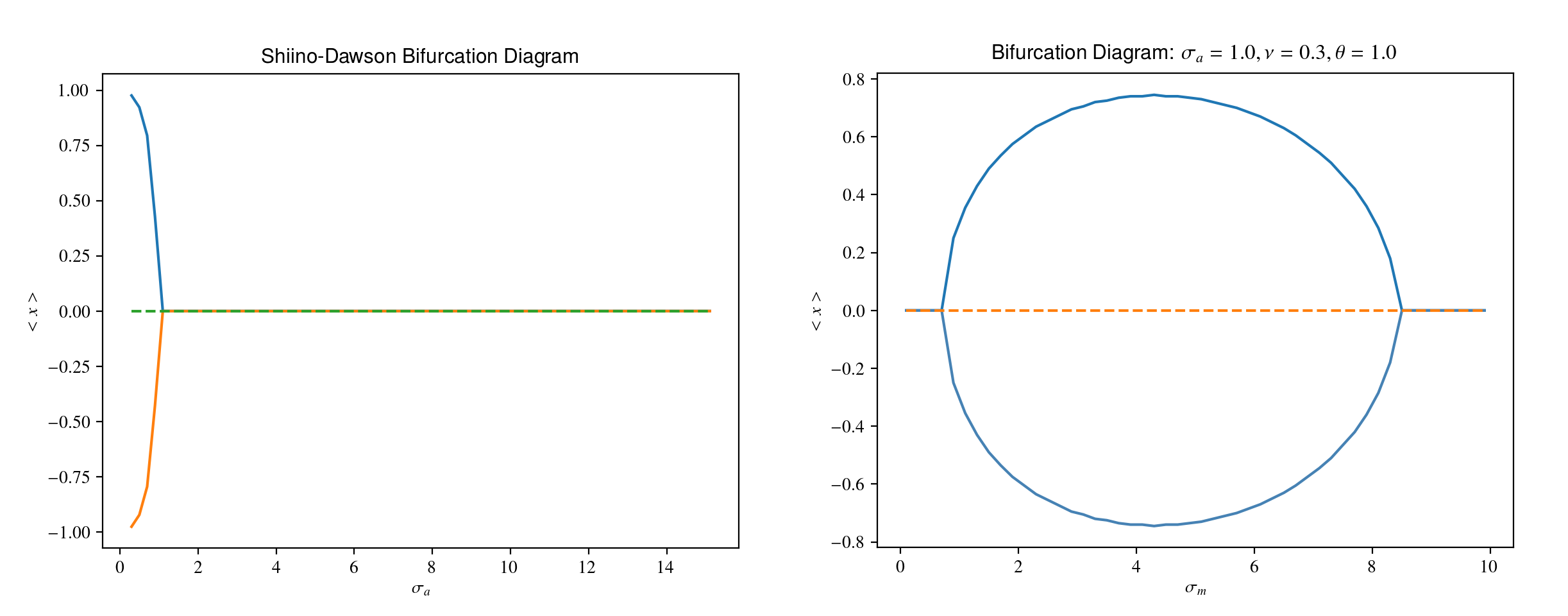}
		\caption{Bifurcation diagrams of (left) the Dawson-Shiino model with a classic pitchfork shape, and (right) the model with uncertain robustness (parameters as inscribed), introduced in section \ref{sec::model}}
		\label{fig::bif2}
	\end{figure}

	As to potential and drift choice, it was recently shown in \cite{alecio} that MV-SDE (\ref{vanpr}) has identical phase structure for a broad class symmetric bistable potentials, increasing drift and reversion-type cooperation.
    
    Heuristically, the mechanism of this change is simple. In the stable phase the cooperative terms dominate, and probability mass is concentrated, settling in a single potential well. As noise intensity is increased, the mass outside the well increases relatively and the mean approaches 0. At this point, the potential barrier is overwhelmed and the other well is equally filled, and these solution fold into the symmetric stationary measure at $\sigma_c$. (While the symmetric stationary measure exists in the stable phase of MV-SDE (\ref{vanpr}), its basin of attraction comprises only symmetric initial conditions \cite{alecio2}, so can be ignored). The potential well in which the empirical mean $\bar x$ is located is identified as the system state.

    On the other hand, $n$-SDE system (\ref{ptype}) has a unique stationary measure irrespective of parameters. While $\bar x(t)$ will remain close to $\int x\rho dx$ by the convergence result, in the stable phase there is a non-zero probability of a system state transition: $\bar x(t)$ transitioning to the other state in finite time, which decreases as $n\rightarrow\infty$. Extensive numerical testing in \cite{ss,gomes} has validated this, with $\bar x(t)$ remaining in one state for increasing duration as $n$ increases over a fixed time period. In the unstable phase, transitions between the symmetric wells/states become so common, the mean is 0.
    
    For transitions between states to be meaningful then, $\sigma$ must be fixed in order for the system to be in the stable state, as reasoned in \cite{ss}. \cite{ss} proceeds to study systems with component dependent cooperation intensities and the probability of system state transition, or systemic risk, using large deviation results of \cite{dawgart} along with various linearisations. They were able to show increased cooperation can lower the risk of an individual component failing, but with the risk of systemic failure, in accord with empirical observation, providing further corroboration this choice of this model and cooperation mechanism. 
	
	\section{The Mean-Field Model with Uncertain Stability}
	\label{sec::model}

	In this work, we consider a simple, though ultimately non-trivial, modification to the intrinsic stability. As a starting point, consider $V$ taken from parametric family of symmetric bistable potentials $\frac{x^4}{4}-a\frac{x^2}{2}$, $a>0$. As $a$ increases, so does the height of the potential barrier between states.
    
    In \cite{ss}, the system's stability is intuitively identified to be the resistance of $\bar x(t)$ to changing state. This is
    dependent on the stability of individual components at the microscopic level, equivalently their resistance to the stochastic perturbation changing their state, which is itself a function of aggregating factors such as the size of the potential barrier between risk states, $a$ and strength of cooperative terms, $\theta$, see for instance \cite{gardiner, hangi, PGDiff}

    Lifting these ideas to the macroscopic level, MV-SDE (\ref{vanpr}) is stable if it is in the stable phase, and so $\int x\rho\neq0$ and distinct system states exist. The system becomes more stable with respect to a change in parameter if the size of its stable phase (range of noise strength $\sigma$ such that MV-SDE (\ref{vanpr}) is in the stable phase) increases. This definition is first proposed in \cite[p.157]{ss}. That this definition is in accord with the microscopic was shown in \cite{alecio}, which demonstrated that the system becomes more stable as the aggregating factor of cooperation strength $\theta$ increases. The same result will be presented here for $a$.  

     Suppose now there is some uncertainty in the height of the potential barrier between the risk states, by replacing $a$ with a stochastic process driven by an independent Wiener process for each component: $a \rightarrow a + \sigma_m dB^{(2,i)}_t$. This could represent an incomplete state of knowledge of the implicit stability of individual agents, but can also be physically motivated. Returning to our original examples, the robustness of industrial components can be undermined by thermal fluctuations, which can be represented stochastically. Banks remain solvent when their liabilities are outweighed by their assets. These assets will be invested and their value dependent on market forces; downward movements can leave banks vulnerable to failure; asset price contagion \cite{bsc2,bsc}. In this case, the risk state is a measure of their liabilities and $a$ the initial value of their assets, with the diversity of fluctuations reflecting the differing assets each bank holds.


    Replacing $\sigma$ with $\sigma_a$, and substituting for $a$ in (\ref{ptype}), the associated MV-SDE is
    \begin{equation}
		\label{funda1}
		dX_t=\big(aX_t-X_t^3-\theta(X_t-m_1)\big)dt+\sigma_adB_t^{(1)}+\sigma_m X_t\circ_\nu dB_t^{(2)}
	\end{equation}
    The stochastic integral of the second Wiener process is open to multiple interpretations. This is denoted by $\circ_\nu$ with $\nu\in[0,1]$, determining where the value of the integrand is sampled in the limiting Riemann sums. It is well known that the lack of regularity of the Wiener process leads to entirely different values of the integral, for non-trivial integrand. The most commonly used stochastic integrals -  Klimontovich, Stratonovich and It\^o - correspond to $\nu=\{0,\frac{1}{2},1\}$. (The It\^o integral will also be denoted by omitting the $\circ$) Realisations of these stochastic integrals are known to occur in nature, \cite{PesceGiuseppe2013Stin}. Aside from empirical observation, factors influencing choice of $\nu$ are considered in Section \ref{sec::res}.

	In this work it will be shown that, while the straightforward competition of total noise to aggregating factors remains, increasing the total noise by increasing the multiplicative noise, representing increased uncertainity in components robustness, has a far more varied effect. It can destabilise the system, as might be expected, but can be a neutral factor or even influence an unstable system back into stability, so-called noise induced stability.

	\section{Mathematical Formulation}

    MV-SDE (\ref{funda1}) is equivalent in law to 
    \begin{equation}\label{funda}
		dX_t=\big(aX_t-X_t^3+X_t-\theta(X_t-m_1)\big)dt+\sqrt{\sigma_a^2+\sigma_m^2X^2_t} \circ_\nu dW_t
	\end{equation}
    In terms of the It\^o stochastic integral, MV-SDE (\ref{funda}) is
	\[dX_t=\big(aX_t-X_t^3+(1-\nu)\sigma_m^2 X_t-\theta(X_t-m_1)\big)dt+\sqrt{\sigma_a^2+\sigma_m^2X^2_t}dW_t\]
	As phase transitions and stability are entirely discernible from the law of the process, MV-SDE (\ref{funda}) is the fundamental object of study in this work. 
    
    The concomitant Fokker-Planck equation is
    \begin{equation}\label{fpe}
		\frac{\partial}{\partial t}\rho=\frac{\partial}{\partial x}\Big((-x+x^3-(1-\nu)\sigma_m^2x-\theta(m_1-x))\rho+\frac{1}{2}(\sigma_m^2x^2+\sigma_a^2)\frac{\partial\rho}{\partial x}\Big)
	\end{equation}
    This specific model was first introduced in \cite{Multinoise}, with $\nu=\frac{1}{2}$, and has been the subject of recent interest in \cite{agp}.

    Directly integrating (\ref{fpe}), the general form of the stationary measure is
    \begin{equation}\label{statmom1}
		\rho_{0}=\exp\big( \frac{a-\theta-\nu\sigma_m^2+\frac{\sigma_a^2}{\sigma_m^2}}{\sigma_m^2}\log(1+\frac{\sigma_m^2}{\sigma_a^2}x^2)+\frac{2\theta\mu}{\sigma_a\sigma_m}\arctan{\frac{\sigma_mx}{\sigma_a}}-\frac{x^2}{\sigma_m^2}\big)
	\end{equation}
    where $m_1=\int x\rho[m_1] dx$. These correspond to the roots of the self-consistency function
    \begin{equation}\label{sfcn}
		F(\nu, \sigma_a,\sigma_m,a,\theta)[\mu]=\int (x-\mu)\rho_{0}(\nu, \sigma_a,\sigma_m,a,\theta)[\mu]dx
	\end{equation} 
    which is a more appealing form for technical reasons \cite{alecio}. The roots of $F[\mu]$ are not necessarily unique, translating to multiple admissible stationary measures.

		\begin{figure}[b]
		\centering
		\includegraphics[width=\textwidth]{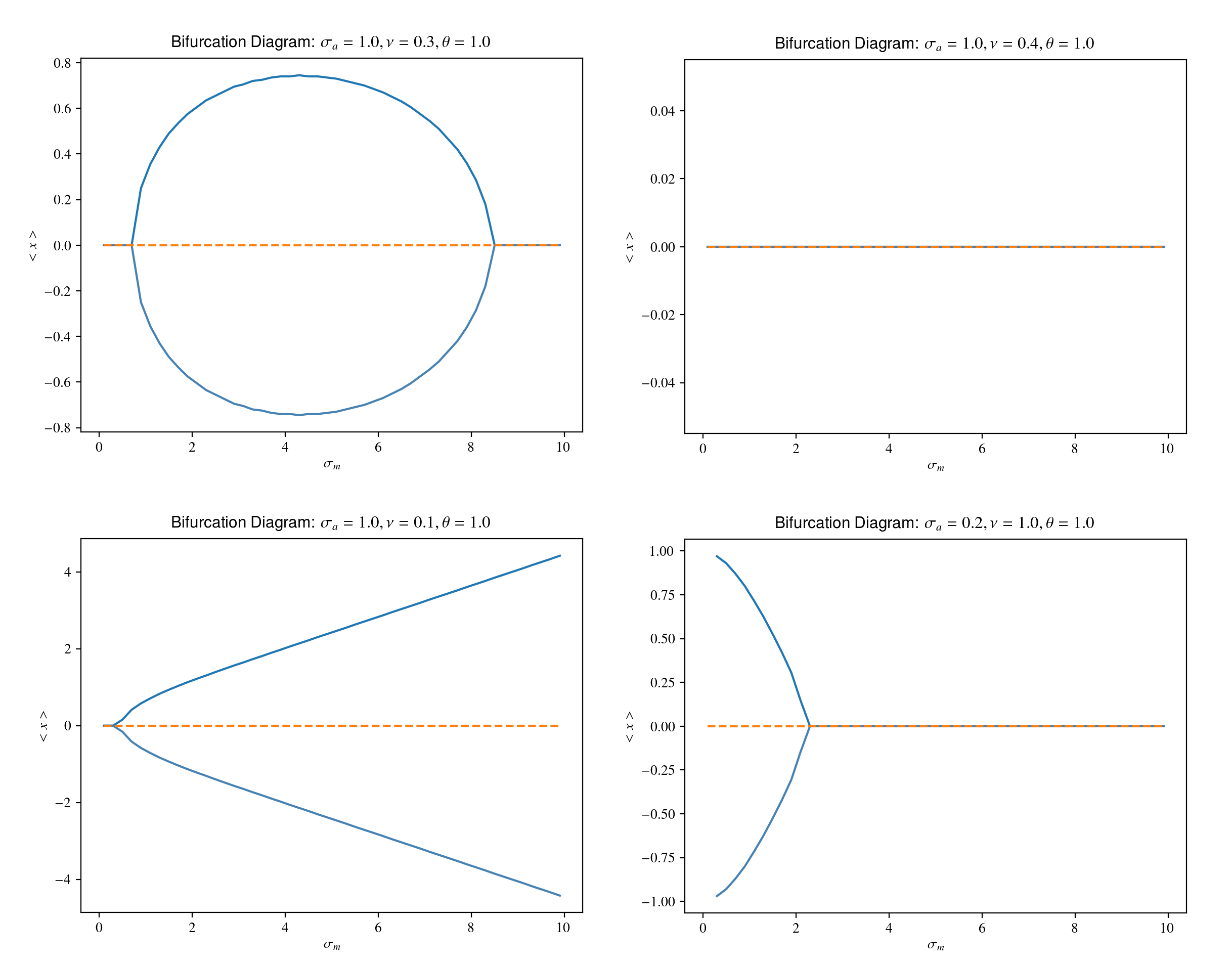}
		\caption{Panel of bifurcation diagrams, parameters inscribed. Left to right, top to bottom: $1\rightarrow3\rightarrow1$, $1\rightarrow$, $1\rightarrow3$ and $3\rightarrow 1$ for $(\nu,\,\sigma_a)$ as inscribed}
		\label{fig::bif}
	\end{figure}

	The following results, adapted from \cite{alecio}, expatiate the relationship between phase and self-consistency function $F$ of MV-SDE (\ref{funda}), sketching $F$ on rays in $(\sigma_a,\sigma_m)$-space, where the multiplicative and additive noise increase in intensity in fixed ratio $\sigma_m=k\sigma_a$, with varying $a$ and $\theta$. The interested reader can investigate their technical underpinning and precise conditions in \cite{alecio}, with any pertinent additional information relegated to Appendix \ref{sec::proofss}. 

	For MV-SDEs with elliptic drifts of the form $\sigma k(X_t)dW_t$ it was shown in \cite{alecio} there can only be 1 or 3 stationary measures, demarcating the stable and unstable phase. Its direct analogue can be concluded for MV-SDE (\ref{funda}) including, crucially, that the stability of MV-SDE (\ref{funda}) directly corresponds to the sign of $F^{'}_\mu(\nu,\sigma_a,\sigma_m,a,\theta)[0]$. 
	 
	\begin{proposition}[\cite{alecio} Proposition 3.3]
		\label{fprop}
	MV-SDE (\ref{funda}) has two phases, stable and unstable, characterised by possessing 3 (respectively 1) stationary measures. It is in the stable phase iff $F^{'}_\mu[0]>0$
	\end{proposition}
	
	The next shows the aggregating factors work, as for MV-SDE (\ref{vanpr}), to make MV-SDE (\ref{funda}) more stable:
	\begin{proposition}[\cite{alecio}] 
	\label{stabres} If $\sigma_m=k\sigma_a$, MV-SDE (\ref{funda}) is more stable as $a$ or $\theta$ increases.
	\end{proposition}
	\begin{proof}
		Appendix \ref{sec::proofss}
	\end{proof}

	The last concerns the phase structure, with the It\^o integral. The rigidity of the phase structure, stable to unstable (or $3\rightarrow 1$ in shorthand), is a characteristic feature of MV-SDE (\ref{vanpr}) with symmetric potential and diffusion.

	\begin{proposition}[\cite{alecio} Proposition 3.5]
	\label{starshaped}
	If $\sigma_m=k\sigma_a$ and $\nu=1$, MV-SDE (\ref{funda}) will transition from the stable to unstable phase $(3\rightarrow1)$ as $\sigma_a$ increases.
	\end{proposition}

	\begin{figure}[b]
		\centering
		\includegraphics[width=.9\textwidth]{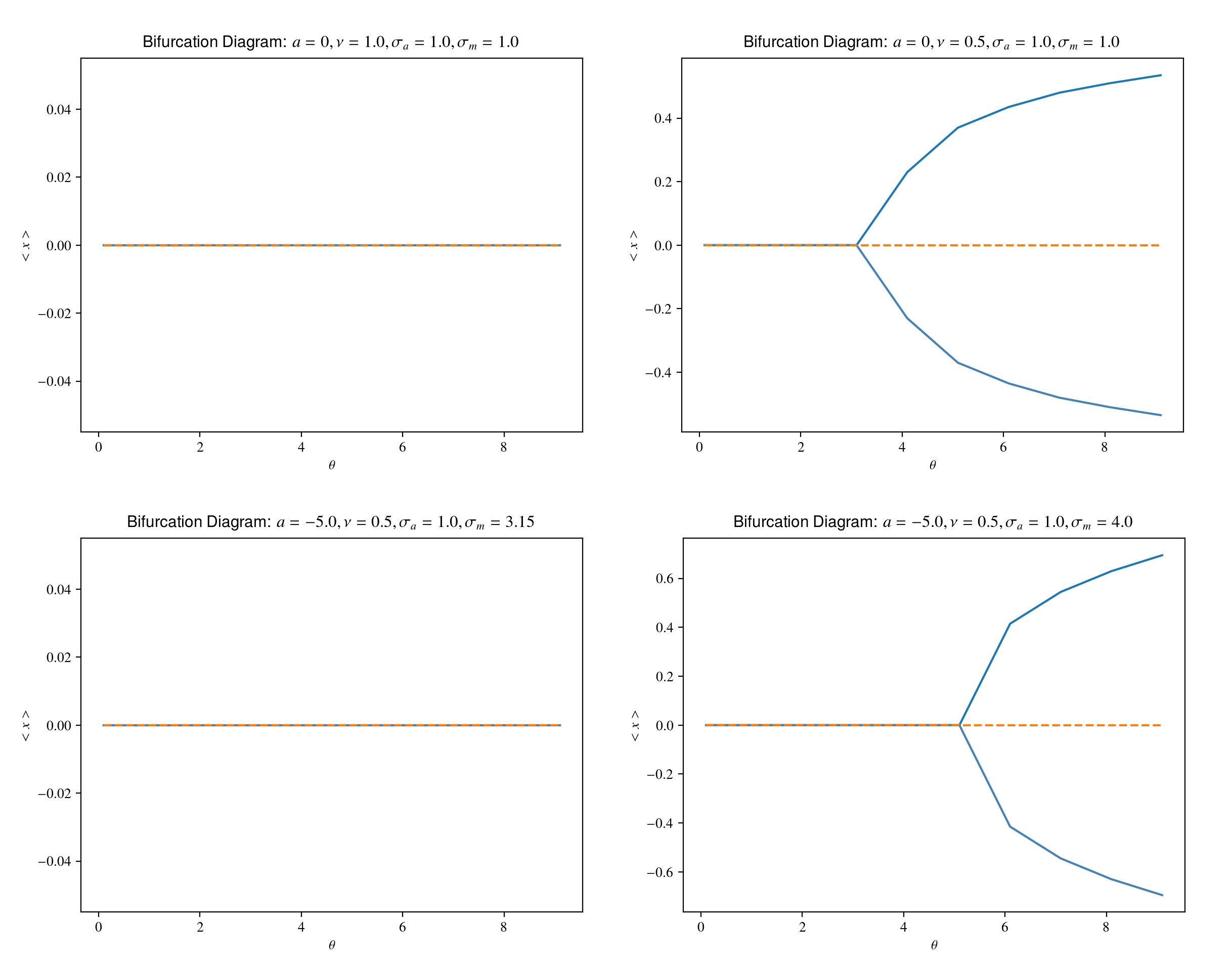}
		\caption{Bifurcation Diagrams for $a\leq0$. \textit{Top} At $a=0$ a stable phase exists so long as $\nu=1$. \textit{Bottom} Stable phase, and lack thereof, above and below the threshold. Note $\sqrt{10}\approxeq3.16$}
		\label{fig::bfp2}
	\end{figure}
	
	It is both the noise interpretation and uncertainty in components' robustness (equivalently, the ability to vary both multiplicative and additive noise out of ratio) that diversifies this phase structure, underpinning the results of this work. As an example, for some range of $\nu$, multiplicative noise can make the system more stable, as the next section will show.

	\section{Noise Induced Stability}
	\label{sec::res}
	This section will the effect of multiplicative noise on phase structure of MV-SDE (\ref{funda}). Particularly, it will be shown that, depending on $\nu$, increased multiplicative noise can actually transition the system to the stable phase and its presence can even permit the existence of a stable phase where none can exist without.

	A clear example of the latter, and one the results of \cite{alecio} are particularly well disposed to study, is when $\theta$ is varied for fixed $(\sigma_a,\sigma_m)$.
	If the potential is bistable, the phase structure is not altered by multiplicative noise: by Proposition \ref{stabres}, MV-SDE (\ref{funda}) will be stable for sufficiently large $\theta$, regardless of $(\sigma_a,\sigma_m)$. 
	Contrastingly, if the potential were convex, and $\nu=1$, there is only one phase. It will be demonstrated that multiplicative noise can induce a stable phase, which could not otherwise exist.

	Concretely, consider MV-SDE (\ref{funda}) again with the potential $V^{'}=x^3-ax$, where now $a\in\mathbb{R}$. If $a>0$, the potential is bistable and Proposition \ref{stabres} applies. When $a\leq0$, the three extrema merge, forming one minima at 0. 
	It is known that the number of stationary measures is dependent on the number of extrema. Indeed, it is straightforward to retool the results of the second section of \cite{alecio} to achieve
	\begin{theorem}[\cite{alecio} Theorem 2.12]
		\label{magnatheorem}
		Suppose $\nu=1$, and $V^{'}$ has $N$ roots, all simple. Then there exists $\theta_c$ such that for $\theta>\theta_c$ MV-SDE (\ref{funda}) has $N$ stationary measures.
	\end{theorem}
	In fact, for any convex potential in the presence of additive noise (or for a more general diffusion with the It\^o integral) there is a unique stationary measure and consequently one phase, see also \cite{malrieu}.

	With multiplicative noise and non-It\^o integral ($\nu\neq1$), $V^{'}$ is augmented by the integral correction term $(1-\nu)\sigma_m^2x$, with extensive ramifications. As before, for sufficiently small $\theta$ and large $\sigma_a$, it can be shown the stationary measure is unique. However, applying Theorem \ref{magnatheorem} to $V^{'}=x^3-ax-(1-\nu)\sigma_m^2x$ at $a=0$, any level of multiplicative noise allows for stable phase with a phase change (unstable to stable) as $\theta$ is increased. Similarly for $a<0$ the same result holds so long as $(1-\nu)\sigma_m^2>-a$.

	It is tempting, then, to conclude that multiplicative noise is always a stabilising influence, given that it deepens the potential wells. In fact, multiplicative noise also increases the weight of the tails of the stationary measure(s), which is destabilising. It is the competition between these two elements that will be the subject of the sequel, by study of $F^{'}_\mu[0]$. As Proposition \ref{stabres} establishes their effect, unless otherwise stated, $a$ and $\theta$ are set to unity in the following. $\sigma_c$ denotes the critical temperature of the Dawson-Shiino model, MV-SDE (\ref{vanpr}) with the simple bistable potential.\newline

	\begin{center}
		
		\begin{tabular}{   |p{2.2cm}||p{2.2cm}|p{2.2cm}|p{2.2cm}|p{2.2cm}| }
			\hline
			\multicolumn{5}{|c|}{Phase Portrait Summary} \\
			\hline
			
			$\nu$ & $\nu>\nu_1$ &$\nu_1>\nu>\nu_2$ &$\nu_2>\nu>\nu_3$& $\nu_3>\nu>0$\\
			\hline
			$\sigma_a $  & $3\rightarrow1$    &$ 3\rightarrow$&$ 3\rightarrow$&$3\rightarrow$\\
			increasing&   $1\rightarrow$  & $3\rightarrow1$&$1\rightarrow3$&$1\rightarrow3$\\
			$\downarrow$& &$1\rightarrow3\rightarrow1$&$1\rightarrow3\rightarrow1$ &$1\rightarrow$\\
			& & $1\rightarrow$&$1\rightarrow$&\\
			\hline
			&\multicolumn{4}{|c|}{$\sigma_m$ increasing $\rightarrow$} \\
			\hline
		\end{tabular}
	\end{center}

	\begin{center}
		\qquad\quad Table 1: Phase Transition Summary. Correspondence to figure \ref{fig::cont} described below the figure\newline
	\end{center}
	

	
	In Figure \ref{fig::cont}, a panel of contour graphs of $F_\mu^{'}[0]$, is presented, with the phase transition contour $F^{'}_\mu=0$, for a representative range of $\nu$. The phase changes for increasing $\sigma_m$, are presented in Table 1.
	
		
	%
	
	From Proposition \ref{starshaped}, with It\^o noise the set $\{(\sigma_a,\sigma_m):F^{'}_\mu[0]>0\}$ is star shaped about the origin. Broadly, as $\nu$ decreases, phase transition contour, parameterised by $\sigma_m$, moves away from the origin. The critical change is that below  $\nu_1$, the shape begins to change, from decreasing to increasing before decreasing, as $\sigma_m$ increases. Below $\nu_3$, it is increasing. 

	\begin{proposition}[Asymptotic Properties of $ F^{'}_\mu(\nu,\sigma_a,\sigma_m){[}\mu{]}$: $\sigma_m\downarrow 0$ ]$ $\newline
		\label{pm2}
		Let $G$ be the self-consistency function and $m_i$ the $i^{\mathrm{th}}$ moment of the Dawson-Shiino model.
		\begin{enumerate}
			\item $\lim\limits_{\sigma_m\downarrow 0} F^{'}_\mu(\nu,\sigma_a,\sigma_m)[0]=G^{'}(\sigma_a)[0]$
			\item $\lim\limits_{\sigma_m\downarrow 0}\frac{\partial F^{'}_\mu}{\partial\sigma^2_m}(\nu,\sigma_a,\sigma_m)=\frac{m_8-m_{10}}{3\sigma_a^4}+\frac{m_6-m_4}{3\sigma_a^2}-(1-\theta)\frac{m_6-m_8}{\sigma_a^4} +m_2-\nu(\frac{m_4-m_6}{\sigma_a^2}+m_2)$
		\end{enumerate} 
	\end{proposition} 
	\begin{proof}
		Appendix \ref{sec::proofss}
	\end{proof}

	
	By the first point of the above proposition, the phase transition contour must emanate from $(\sigma_a,0)$. Further, knowing that $\frac{\partial G^{'}}{\partial \sigma_a}(\sigma_c)<0$ by Proposition 3.5 of \cite{alecio}, the sign of the gradient of the phase transition contour is equal to that of $\lim\limits_{\sigma_m\downarrow 0}\frac{\partial F^{'}_\mu}{\partial\sigma^2_m}(\nu,\sigma_c,\sigma_m)$ by the chain rule. This was determined as a function of the first 5 even moments of the stationary distribution of the Dawson-Shiino model in the second point of Proposition \ref{pm2}. In Appendix \ref{sec::proofss}, this was simplified to $\frac{1}{2}(\frac{1}{2}-\nu)(1-m_2)$, a decreasing function in $\nu$ that becomes positive at $\nu=0.5$. See Figure \ref{fig::cor2} for this sign change's dependence on $\theta$.

	\begin{figure}[h]
		\centering
		\includegraphics[width=\textwidth]{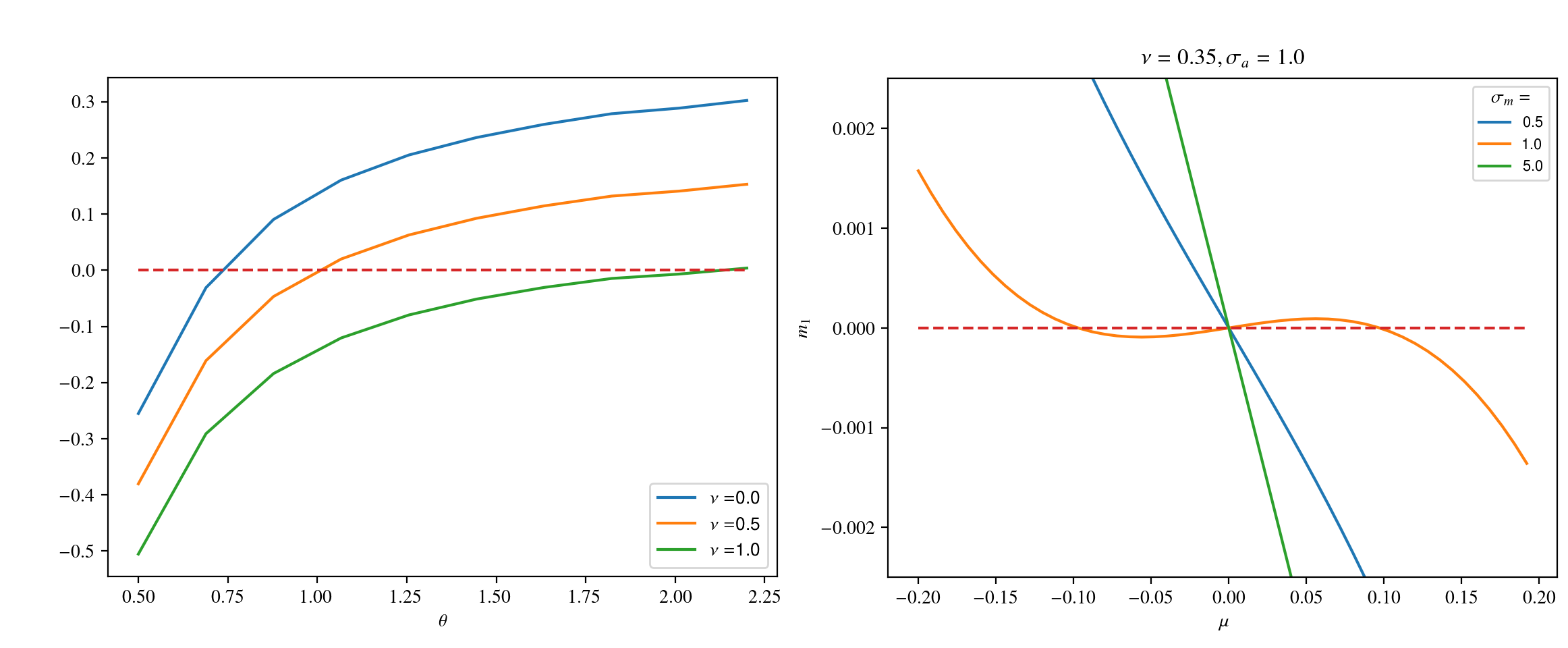}
		\caption{\textit{Left} Gradient of the phase transition contour at $(\sigma_c,0)$ against $\theta$ for It\^o, Stratonovich and Klimontovich noise. The roots at $\theta=1$ for $\nu=0.5$ has been recovered and those for $\nu=0$ and $\nu=1$ displayed. For $\theta$ above $\sim0.72$ for some range of $\nu$, noise induced stabilisation can be observed. It always occurs (regardless of $\nu$) for $\theta\gtrsim2.1$. \textit{Right} the self-consistency function for MV-SDE (\ref{funda}) displaying a $1\rightarrow3\rightarrow1$ phase change.
		\label{fig::cor2}}
	\end{figure}

	In the case that at the critical temperature, there is sufficient scale separation that the coefficient process $a + \sigma_m dW_t^2$ can be averaged out, the phase transition contour would be perpendicular to the $\sigma_a$ axis, suggesting the choice of Stratonovich noise, $\nu=0.5$ for very small $\sigma_m$.
	

	Consequently, for $\nu>0.5$ and $\sigma_a$ sufficiently close to $\sigma_c$, the system will transition from unstable to unstable, $(1\rightarrow3)$, and noise induced stabilisation occurs. Whether it returns to instability depends on the properties of $F^{'}_\mu[0]$ as $\sigma_m$ is increased. It can be expected that the limit $\lim\limits_{\sigma_m\uparrow\infty}F^{'}_\mu[0]$ is dependent on $\nu$.
	Indeed, the multiplicand $(1+x^2)^{-\nu}$ in $\rho_0$ dominates in the limit, by decreasing the relative weight of the tails with $\nu$.

	\begin{proposition}[Further Asymptotic Properties of $ F^{'}_\mu(\nu,\sigma_a,\sigma_m){[}\mu{]}$: $\sigma_m\uparrow\infty$ ]$ $
		\label{pm}

			\item If $\nu>0.5$, $\lim\limits_{\sigma_m\rightarrow \infty}F^{'}_\mu(\nu,\sigma_a,\sigma_m)<0$.\\
			Else if $\nu\leq0.5$, $\exists\, \sigma_c^\nu$ s.t for $\sigma_a>\sigma_c^\nu,$  $\lim\limits_{\sigma_m\rightarrow \infty}F^{'}_\mu(\nu,\sigma_a,\sigma_m)<0$ \\and for $\sigma_a<\sigma_c^\nu,$  $\lim\limits_{\sigma_m\uparrow \infty}F^{'}_\mu(\nu,\sigma_a,\sigma_m)>0$
		
	\end{proposition}
	
	\begin{proof}
		Appendix \ref{sec::proofss}
	\end{proof}

	In Appendix \ref{sec::proofss}, the sign of the limit was numerically determined in (\ref{gchange0}). The value of $\sigma_a$ at which it changes sign is
	\begin{equation}\label{sgnff}\sigma_a=\frac{\pi\Gamma(1 - \nu)}{\Gamma(1/2 - \nu)}\end{equation}
	which is a decreasing function in $\nu$, with a root at $\nu=0.5$. With this, and the previous determination that the phase transition contour is tangential to the $\sigma_a$ axis at the same value of $\nu$, $\nu_3=0.5$.
	
	Solving (\ref{sgnff}) at $\sigma_a=\sigma_c$ yields $\nu\approxeq0.28$. For $\nu$ below this, the phase transition $3\rightarrow 1$ cannot exist for $\sigma_a<\sigma_c$. Therefore $\nu_2\approxeq0.28$
	

	By similar reasoning, for $0.28<\nu<0.5$ and $\sigma_a$ greater than but sufficiently close to $\sigma_c$, the line of constant $\sigma_a$ must intersect the contour at least twice, corresponding to the phase change $1\rightarrow3\rightarrow1$: unstable to stable, returning to unstable again (see Figure \ref{fig::bif}, top left graph). $\nu_1$ is in fact less than 0.28 as the asymptote of the phase transition contour is still smaller than its peak. This was determined numerically to be $\nu_1\approxeq0.11$. Below this point the contour is seen to be strictly increasing, limiting possible phase changes further.
	
    \section{Conclusions}
	In this work an MV-SDE with bistable drift, with additive and, novelly, multiplicative noise has been studied.
	After a brief review of systemic risk, following \cite{ss}, a MV-SDE (Dawson-Shiino) model derived from a interacting diffusion model of systemic risk of interconnected components is presented. A range of scenarios where uncertainty in the robustness of the components may occur is discussed, and a novel MV-SDE model is derived.
	
	For this model, the results have demonstrated the existence phase changes that cannot occur in the Dawson-Shiino model, that stem directly from varying noise interpretations and uncertainty in the robustness of components. Of particular interest, for a range of $\theta$ a noise induced stability phenomenon was observed. Namely, if the additive noise $\sigma_a$ is set greater than, but sufficiently close to the critical temperature $\sigma_c$ of the limiting Dawson-Shiino model and an appropriate noise interpretation chosen, increasing multiplicative noise $\sigma_m$ will push the system into the stable phase. It will remain there, or re-enter the unstable phase depending again on $\nu$. 

	A potential future are of inquiry would be whether similar noise induced stability can be seen in MV-SDE (\ref{funda}) with a multi-well potential, see section 4 of \cite{alecio}

	\section{Acknowelgement}
	The initial idea to study MV-SDE (\ref{funda}) numerically was Prof. G.A Pavliotis'.

	\begin{figure}[H]
		\centering
		\includegraphics[width=\textwidth]{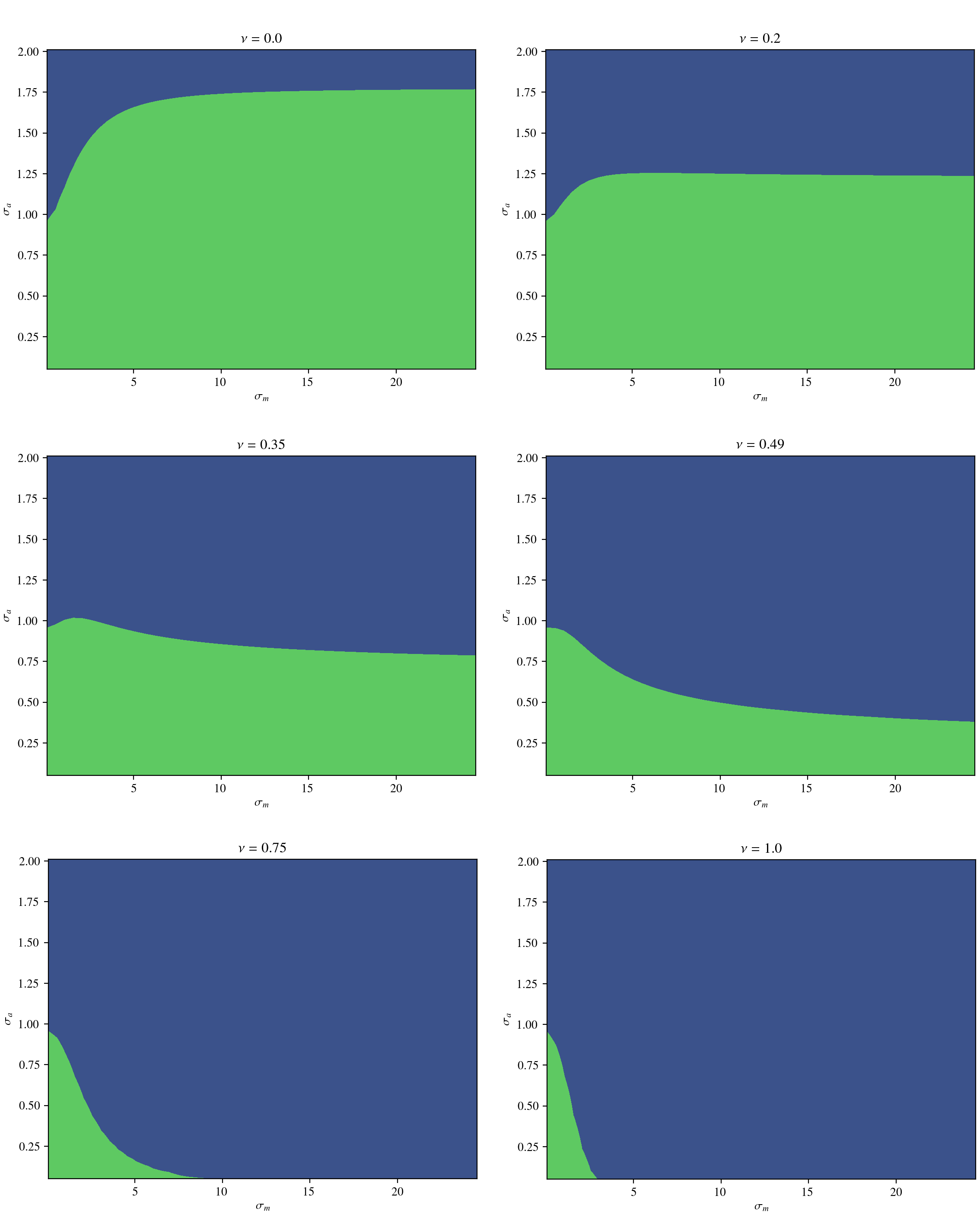}
		
		\caption{Panel of contour diagrams of $F^{'}(\nu,\cdot,\cdot){[}0{]}$ for increasing $\nu$, green positive, blue negative. The phase transition contour intersects the $\sigma_a$ axis at $\sigma_c$. Graph $\nu=\{0.75,1\}$ corresponds to column 4, $\nu=0.49$ and $\nu=0.35$ column 3, $\nu=0.2$ column 2 and $\nu=0$ column 1 of Table 1}
		\label{fig::cont}
	\end{figure}

	\appendix
	\section{Formal Identification of the Limit}
	\par Given the exchangeability and weak interaction between particles (inversely proportional to the number of particles), it seems reasonable to impose as an ansatz that the particles are identically and independently distributed, $\rho_n\approx\prod_{l=1}^n\rho(x_l,t)=\rho^{\otimes n}$ as for $n$  sufficiently large.
	
	The associated Fokker-Planck equation for the $n$-particle system is
	\begin{equation}
		\label{fpfd}
		\frac{\partial\rho_N}{\partial t}=\sum_{i=1}^{N}\frac{\partial}{\partial x_i}\big( V'(x_i)+\theta(x_i-\frac{1}{N}\sum_{j=1}^{N}x_j)\big)\rho_N+\sum_{i=1}^{N}\frac{\partial^2}{\partial x_i^2}\big(\frac{\sigma_a^2+\sigma_m^2 x_i^2}{2}\big)\rho_N
	\end{equation}
	To find a closed expression for $\rho(x_i)=\int_{\backslash i}\rho^N$, we integrate (\ref{fpfd}) over all indices but the $\mathrm{i}^{th}$ - denoted as $\backslash i$.
	
	Consider first the the terms deriving from the drift:
	\begin{equation}
		\begin{split}\int_{\backslash i}\sum_{i=1}^{N}\frac{\partial}{\partial x_i}\big( V'(x_i)+\theta(x_i-\frac{1}{N}\sum_{j=1}^{N}x_j)\big)\rho_N=\frac{\partial}{\partial x_i}\int_{\backslash i}\big( V'(x_i)+\theta(x_i-\frac{1}{N}\sum_{j=1}^{N}x_j)\big)\rho_N\\+\sum_{\backslash i}\frac{\partial}{\partial x_j}\big( V'(x_j)+\theta(x_j-\frac{1}{N}\sum_{k=1}^{N}x_k)\big)\rho_{ij}(x_i,x_j)|_{x_j=-\infty}^{\qquad \infty}\end{split}\end{equation}
	where we assume the both $\rho_N$ and its first derivative with respect to all its variables decays to 0 sufficiently fast to annihilate all the terms in the second line 
	and sufficient smoothness of $\rho^N$ to commute the integral and derivate. We can simplify the remaining term as follows. 
	$$\frac{\partial}{\partial x_i}\Big(V'(x_i)\rho_i+\theta (1-\frac{1}{N})\big(x_i+\int x_i\rho_i\big)\rho_i\Big)$$where we have used that $\int x_i\rho_i=\int x_j\rho_j$ for any $i,j$. Upon taking the limit $N\rightarrow\infty$ we get:	
	\begin{equation}\label{pt1}\frac{\partial}{\partial x_i}\Big(V'(x_i)+\theta \Big(x_i+\int x_i\rho_i)\Big)\rho_i\end{equation}
	As for the second term, we have
	\begin{equation}
		\begin{split}
			\frac{1}{2}\int_{\backslash x_i}\sum_j\frac{\partial^2}{\partial x_j^2}\sigma^2(x_j)\rho^N=\\\frac{\partial^2}{\partial x_i^2}\frac{\sigma^2(x_i)}{2}\rho_i+\sum_{\backslash i}\frac{\partial}{\partial x_j}\frac{\sigma^2(x_j)}{2}\rho_{ij}(x_i,x_j)|_{x_j=-\infty}^{\qquad \infty}
		\end{split}
	\end{equation}
	The assumptions above are strong enough to ensure all the terms in the last sum are null. Adding the remaining term to (\ref{pt1}), renaming $\rho_i$ as $\rho$ we get (\ref{fpe}) as desired. An almost identical calculation can be done in the presence of non-It\^o noise, for a suitable correction in the drift.
	\cite{infe}\newline
	\begin{figure}[t]
		\centering
		\includegraphics[width=.9\textwidth]{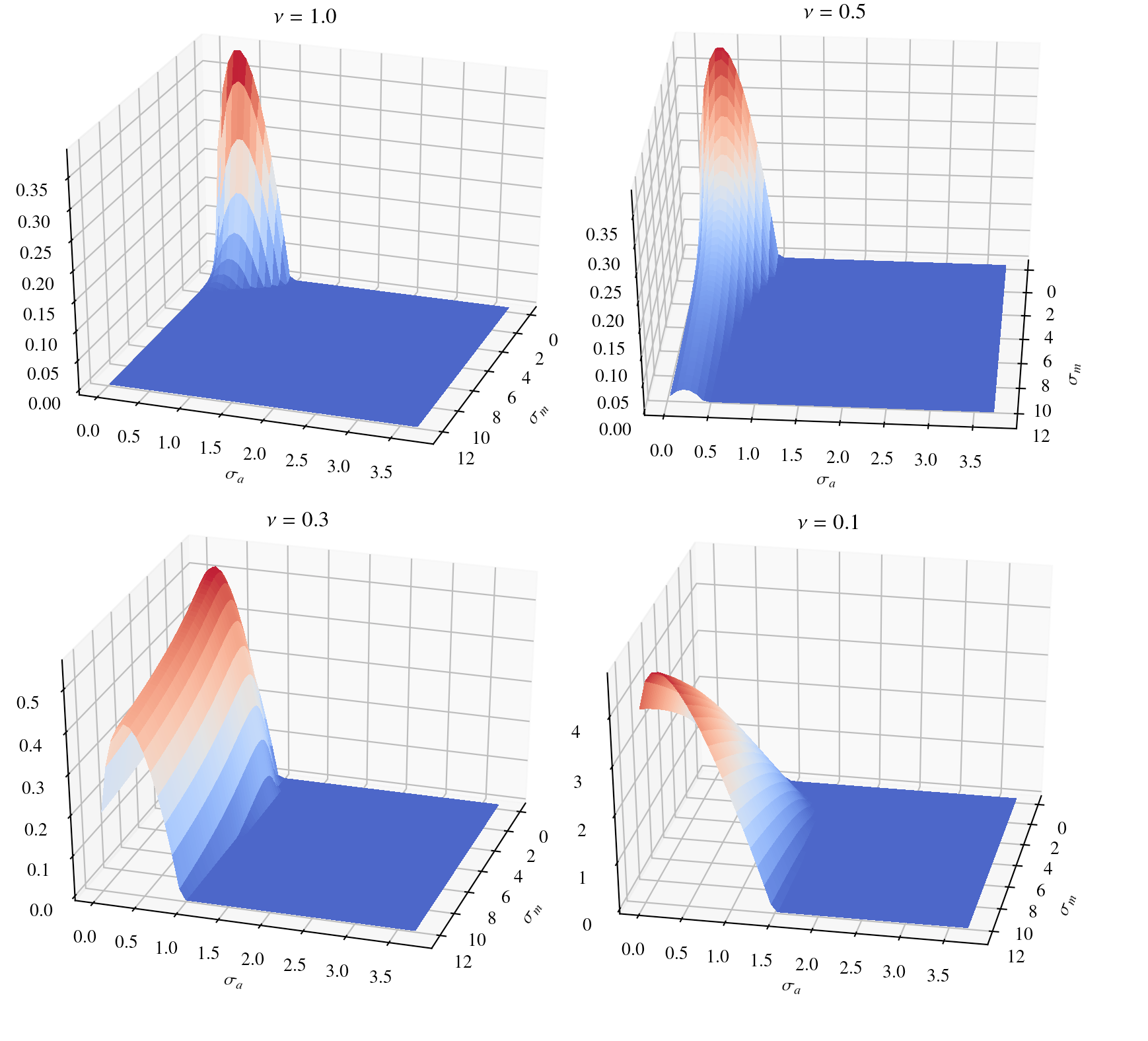}
		\caption{Panel of Bifurcation diagrams for $\nu$ as inscribed}
		\label{fig::3dbif}
	\end{figure}
	
	\section{Proofs}
	\label{sec::proofss}
	\subsection{Proposition \ref{stabres}}
	
	It can be shown at a root of $F^{'}_\mu[0]$, the derivative with respect to $\theta$ or $a$ must be positive, similarly to Proposition 3.5 \cite{alecio}. Then, like Proposition 3.8 of \cite{alecio}, the interval(s) on which $F^{'}_\mu(\sigma_a,k\sigma_m)[0]>0$ must be increasing.

	\subsection{Proposition \ref{pm2}}
	
	The idea here is to expand the self-consistency equation, paying close attention to their radius of convergence, whilst recovering the Shiino-Dawson symmetric stationary measure. The radius of convergence of both $\arctan \sigma_mx$ and $\log(1+\sigma_mx)$ are finite, limiting the domain of the resulting integral:
	\begin{equation}\label{intexp}
	2\int_0^{\frac{1}{\sigma_m^2}}\exp(-\frac{x^4}{2\sigma_a^2})((x-x^3)+\sigma_m^2(1-\nu)x)(x-\frac{\sigma_m^2x^3}{3\sigma_a^2}+\dots)(1+\frac{\sigma_m^2}{\sigma_a^2}(\frac{x^6}{3}-\nu x^2)+\dots)dx
	\end{equation}
	On $(\frac{1}{\sigma_m^2},\infty)$ the above integral is dominated by $k\exp(-\sigma_m^2)$
	Consequently, integral (\ref{intexp}) can be extended to $\infty$, yielding the following expression in the moments of $\rho_0$, to order $\sigma_m^2$
	
	\begin{equation}
		\label{sgnch}
		(m_2-m_4)+\sigma_m^2\Big(\frac{m_8-m_{10}}{3\sigma_a^4}+\frac{m_6-m_4}{3\sigma_a^2}-(1-\theta)\frac{m_6-m_8}{\sigma_a^4} +m_2-\nu(\frac{m_4-m_6}{\sigma_a^2}+m_2)\Big)+\sigma_m^4(\dots)
	\end{equation}
	where a factor of $\sigma_a^2$ has been eliminated. Using the moment hierarchy \cite{dawson} of the symmetric stationary measure, (\ref{sgnch}) can be written entirely in terms of $\sigma_a$ and $m_2$, see (\ref{sympy}) and Appendix \ref{sec:mhel} for its derivation.
 	
	The $\mathcal{O}(1)$ term is just the self-consistency equation of the Dawson-Shiino model, and so is 0 at $\sigma_a=\sigma_c$. For $\theta=1$, at the critical temperature (\ref{sgnch}) is $$\frac{1}{2}(\frac{1}{2}-\nu)(1-m_2)$$
	where $m_2=m_2(\sigma_c^{\theta=1})\approxeq0.457$. This is a decreasing function in $\nu$ with a root $\nu=\frac{1}{2}$. 

	\subsection{Proposition \ref{pm}}
	Using the approach of \cite{alecio}, we rewrite $F^{'}_\mu$ as 
	\begin{equation}\label{scm}
		\frac{4}{\sigma_m\sigma_a}\int_\mathbb{R^+} \arctan(\frac{\sigma_m}{\sigma_a}x)\big(x(1+\sigma_m^2(1-\nu))-x^3\big)\rho_{st}
	\end{equation}
	This form is useful for $\sigma_m<<1$. For $\sigma_m>>1$ the substitution $y=\frac{\sigma_m x}{\sigma_a}$ yields a far more lucid expression,
	
	\begin{equation}\label{ffsc}
		\begin{split}
			\frac{4\sigma_a}{\sigma_m^3}\int_\mathbb{R^+} \arctan(y)\big(y(1+\sigma_m^2(1-\nu))-y^3\frac{\sigma_a^2}{\sigma_m^2}\big)(1+y^2)^{-\nu}\\\exp(\frac{\sigma_a^2}{\sigma_m^4}(\log(1+y^2)-y^2))dx
		\end{split}
	\end{equation}
	If we can approximate $\arctan(x)$ and $ \log(x)$ with power series, we can evaluate the resulting expression with the following formula.
	With $y\ne-1$ and $x>0$
	\begin{equation}\label{incmoms}
		\begin{split}
			\int_x^\infty x^y \exp(-\frac{\sigma_a^2 x^2}{\sigma_m^4}) dx =\frac{1}{2} \sigma_a^{-y-1}\sigma_m^{2y+2} \Gamma(\frac{y + 1}{2}, \frac{\sigma_a^2 x^2}{\sigma_m^4})\approxeq\\
			\frac{1}{2}\sigma_a^{-y-1}\sigma_m^{2y+2}\Big(\Gamma(\frac{y+1}{2})-\frac{2(\sigma_ax)^{y+1}}{y+1}\sigma_m^{-2y-2}+\mathcal{O}(\sigma_m^{-6-2y})\Big)
		\end{split}
	\end{equation}
	where $\Gamma(x,y)$ is the incomplete Gamma function, \cite{absteg}.
	
	For $\nu<0.5$ we derive the the asymptotic expansion of $F^{'}_{\mu}$ in $\sigma_m$ as follows. The difference of the integral over the entire real line and $\{|x|>1\}$ becomes negligible as $\sigma_m$ tends to $\infty$. On this reduced domain, $(1+x^2)^{-\nu} \sim x^{-2\nu}$, $(1+x^2)^{-\frac{1}{\sigma_m^4}}\sim 1$ and $\arctan(x)\sim 1 -\frac{1}{x}-\frac{1}{3x^3}$. 
	
	Substituting into equation (\ref{ffsc}) and using equation (\ref{incmoms}) to evaluate the resulting expression, in expanding in powers of $\sigma_m$:
	\begin{equation}\label{gchange0}
		0.\sigma_m^{6-4\nu}+\sigma_a^{2\nu-2}[\frac{\pi}{2}\Gamma(1-\nu)+\sigma_a\big(\Gamma(\frac{3}{2}-\nu)-(1-\nu)\Gamma(\frac{1}{2}-\nu)\big)]\sigma_m^{4-4\nu}+\mathcal{O}(\sigma_m^{2-4\nu})
	\end{equation}
	This changes sign when 
	\begin{equation}\label{gchange1}
		\sigma_a=-\frac{\frac{\pi}{2}\Gamma(1-\nu)}{\Gamma(3/2-\nu)-(1-\nu)\Gamma(1/2-\nu)}=\frac{\pi\Gamma(1 - \nu)}{\Gamma(1/2 - \nu)}
	\end{equation}
	a strictly decreasing function in $\nu$, with range $[\sqrt{\pi},0]$. When $\nu\approxeq0.28$ this occurs at $\sigma_c$.

	\section{Moment Hierarchy of the Dawson-Shiino model}\label{sec:mhel}

	As noted in \cite{dawson}, the moments of the stationary measures of the Dawson-Shiino model can be found by solving the moment evolution equation. For the the symmetric stationary solution
    \[
	m_{2p}=(1-\theta)m_{2p-2}+\frac{1}{2}(2p-3)\sigma_a^2m_{2p-4}
	\]
	In terms of the $m_2$, the first 5 even moments are:
	\[
	m_4=m_{2} \left(1 - \theta\right) + \frac{\sigma_a^{2}}{2}\]
	\[m_6=m_{2} \left(\frac{3 \sigma_a^{2}}{2} + \theta^{2} - 2 \theta + 1\right) - \frac{\sigma_a^{2} \theta}{2} + \frac{\sigma_a^{2}}{2}\]
	\[m_8=m_{2} \left(- 4 \sigma_a^{2} \theta + 4 \sigma_a^{2} - \theta^{3} + 3 \theta^{2} - 3 \theta + 1\right) + \frac{5 \sigma_a^{4}}{4} + \frac{\sigma_a^{2} \theta^{2}}{2} - \sigma_a^{2} \theta + \frac{\sigma_a^{2}}{2}\]
	\[\begin{split}m_{10}=m_{2} \left(\frac{21 \sigma_a^{4}}{4} + \frac{15 \sigma_a^{2} \theta^{2}}{2} - 15 \sigma_a^{2} \theta + \frac{15 \sigma_a^{2}}{2} + \theta^{4} - 4 \theta^{3} + 6 \theta^{2} - 4 \theta + 1\right)\\ - 3 \sigma_a^{4} \theta + 3 \sigma_a^{4} - \frac{\sigma_a^{2} \theta^{3}}{2} + \frac{3 \sigma_a^{2} \theta^{2}}{2} - \frac{3 \sigma_a^{2} \theta}{2} + \frac{\sigma_a^{2}}{2}\end{split}\]

	Substituting into (\ref{sgnch}):
	\begin{equation}
		\label{sympy}
		\begin{split}
		m_{2} \left(\frac{\nu}{2} - \frac{1}{4} + \frac{\theta^{2} \nu}{\sigma_a^{2}} - \frac{\theta^{2}}{6 \sigma_a^{2}} - \frac{\theta \nu}{\sigma_a^{2}} + \frac{\theta}{12 \sigma_a^{2}} + \frac{1}{12 \sigma_a^{2}} + \frac{\theta^{4}}{6 \sigma_a^{4}} - \frac{\theta^{3}}{2 \sigma_a^{4}} + \frac{\theta^{2}}{2 \sigma_a^{4}}- \frac{\theta}{6 \sigma_a^{4}}\right)\\ - \frac{\theta \nu}{2} + \frac{5 \theta}{24} + \frac{1}{24} - \frac{\theta^{3}}{12 \sigma_a^{2}} + \frac{\theta^{2}}{6 \sigma_a^{2}} - \frac{\theta}{12 \sigma_a^{2}}
		\end{split}
	\end{equation}

	\bibliography{ggbiblio1}{}
    \bibliographystyle{abbrv}
	\end{document}